\documentclass[11pt]{amsart}
\usepackage{enumitem}
\usepackage{amsmath}
\usepackage{amssymb}
\usepackage{amsthm}
\usepackage{verbatim}
\usepackage{stmaryrd}
\usepackage{xcolor}
\usepackage[utf8]{inputenc}
\newtheorem{thm}{Theorem}[section]

\newtheorem{lem}[thm]{Lemma}

\theoremstyle{definition}
\newtheorem{defn}[thm]{Definition}

\theoremstyle{remark}
\newtheorem{rem}[thm]{Remark}

\newcommand{\nocomma}{}

\newcommand{\tmop}[1]{\ensuremath{\operatorname{#1}}}
\newcommand{\strong}[1]{\textbf{#1}}
\providecommand{\xequal}[2][]{\mathop{=}\limits_{#1}^{#2}}

\title[Sufficient conditions for the existence of LCWs]{Sufficient conditions for the existence of limiting Carleman weights}

\def\R{\mathbb{R}}

\def\2L{\Lambda_{\tilde{\gamma}}}
\def\1L{\Lambda_{\gamma}}

\newcommand{\Lie}{\mathcal{L}}

\renewcommand{\ge}{\geqslant}

\DeclareMathOperator{\Tr}{Tr}

\newcommand{\Implies}[2]{$\text{\ref{#1}}\implies\text{\ref{#2}}$}%
\begin{document}

\author{Pablo Angulo-Ardoy}
\address{ Department of Mathematics, Universidad Aut\'onoma de Madrid}
\curraddr{}
\email{pablo.angulo@uam.es}

\author{Daniel Faraco}
\address{ Department of Mathematics, Universidad Aut\'onoma de Madrid, and ICMAT CSIC-UAM-UCM-UC3M}
\curraddr{}
\email{daniel.faraco@uam.es}

\author{Luis Guijarro}
\address{ Department of Mathematics, Universidad Aut\'onoma de Madrid, and ICMAT CSIC-UAM-UCM-UC3M}
\curraddr{}
\email{luis.guijarro@uam.es}

\thanks{The authors were supported by research grants MTM2011-22612, MTM2011-28198, MTM2014-57769-1-P and MTM2014-57769-3-P from the Ministerio de Ciencia e Innovaci\'on (MCINN), by ICMAT Severo Ochoa projects SEV-2011-0087 and  SEV-2015-0554 (MINECO), and by the ERC 301179
}

\begin{abstract}
In \cite{AFGR}, we found some necessary conditions for a Riemannian manifold to admit a local limiting Carleman weight (LCW), based upon the Cotton-York tensor in dimension $3$ and the Weyl tensor in dimension $4$.
In this paper, we find further necessary conditions for the existence of local LCWs that are often sufficient.
For a manifold of dimension $3$ or $4$, we classify the possible Cotton-York, or Weyl tensors, and provide a mechanism to find out whether the manifold admits local LCW for each type of tensor.
In particular, we show that a product of two surfaces admits a LCW if and only if at least one of the two surfaces is of revolution. This provides an example of a manifold satisfying the eigenflag condition of \cite{AFGR} but not admitting $LCW$.
\end{abstract}

\maketitle
\pagestyle{myheadings}
\markleft{P.ANGULO-ARDOY, D. FARACO AND L. GUIJARRO}

\section{Introduction}

Since the 1987 foundational paper of Sylvester and Uhlmann \cite{SylvesterUhlman87} (for more recent results  see \cite{CaroRogers,HabermanTataru13,Haberman15}), the only effective strategy to solve the Calder\'on inverse problem, 
has been based on the existence of Complex Geometric Optic solutions, CGO for short. In the Riemannian setting, it was discovered in \cite{DKSU07} that CGO solutions depend on the existence of so called limiting Carleman weights.
The existence of such functions was shown to be a problem in conformal geometry:

\begin{thm}[{\cite[Theorem 1.2]{DKSU07}}]\label{DKSU07}
If $(M,g)$ is a open manifold having a limiting Carleman weight, then some conformal multiple of the metric $g$, called $\tilde{g} \in [g]$, admits a parallel unit vector field. For simply connected manifolds, the converse is true.
\end{thm}
For further developments see \cite{CaroSalo14,DKS13,DKLS14,KSU10,SaloLN,Salo13}.
To avoid the simply connected hypothesis, we will focus on existence of a \emph{local} LCW at a point $p$, which is a LCW defined on some neighborhood of $p$. 

Since this condition relates to the conformal class, the paper \cite{AFGR} studied such condition in  terms of the classical tensors of conformal geometry, i.e, the Weyl and Cotton tensors (see section $2$ in \cite{AFGR} for the basics on algebraic curvature operators and bivectors over a real vector space $V$; we will stick to the notation from that paper). In \cite{AFGR} we introduced the following notion:

\begin{defn}[\cite{AFGR}]\label{eigenflag}
 Let $W$ be a \emph{Weyl tensor} in $S^2(\Lambda^2V)$. We say that  $W$ satisfies the \emph{eigenflag condition} if and only if there is a vector $v\in V$ such that $W(v \wedge v^{\perp})\subset v \wedge v^{\perp}$.
\end{defn}

By examining the Weyl tensor of metrics with an $\R$-factor, we obtained in \cite{AFGR} the following obstruction for the existence of LCW's.

\begin{thm}[\cite{AFGR}]\label{Weyl} 
Let $(M,g)$ be a Riemannian manifold of dimension $n\ge 4$.
Assume that a metric $\tilde{g} \in [g]$ admits a parallel vector field.
Then for any $p\in M$, $W_p$ satisfies the eigenflag condition.
In particular, for any $p\in M$, $W_p\in S^2(\Lambda^2(T_p^\ast M))$ has at least $n-1$ linearly independent eigenvectors which are simple.
\end{thm}

In dimension 3, the obstruction was described in terms of the Cotton-York tensor $CY$.

 \begin{thm}[\cite{AFGR}]\label{Cotton}
  Let $n=3$. If a metric $\tilde{g} \in [g]$ admits a parallel vector field, then for any $p\in M$, there is a tangent vector $v\in T_pM$ such that  
  $$
  CY_p(v,v)=CY_p(w_1,w_2)=0
  $$
  for any pair of vectors $w_1,w_2 \in v^\perp$.
 \end{thm}
Both theorems only gave necessary conditions for the existence of LCW's. The results of the current paper study whether the converse to the above results hold. 
Specifically, we provide sufficient conditions for the existence of limiting Carleman weights, and thus CGO solutions, and develop them in detail in dimensions three and four.

For 3- and 4-dimensional manifolds, the results in this paper and in \cite{AFGR} can answer whether the manifold admits a limiting Carleman weight, and identify them if they exist, except for some corner cases that may require some ad hoc  work. We show in section \ref{section: example FH} how to deal with these cases.  

The proof of our results combines a more precise analysis of the algebraic structure of the Weyl and Cotton tensor and an analysis of how distributions in $TM$ get affected
by conformal changes in the metric.  With a slight abuse of notation we define
\begin{defn}\label{defn: conformal factor}
 $D\subset TM$ is a conformal factor of a metric $g$ if it is a smooth distribution of constant rank, such that a conformal multiple of $g$ is a product metric with $D$ and $D^\perp$ tangent to the factors.
\end{defn}

We prove that for a concrete distribution $D$, the behaviour of the Lie derivatives or of the Covariant derivatives characterises conformal factors
 (see our Theorem \ref{prop: characterization of a conformal product factor} ). Thus if for a metric $(M,g)$  there are a finite number of eigenflag conditions the combination of  this result for
 distributions of rank $1$  and  theorems~\ref{Weyl} and \ref{Cotton} answers the question of the existence of LCW

In this way we can analyse all three dimensional manifolds.  
\begin{thm}\label{thm:three dimensions}
Let $(M,g)$ be a 3-dimensional manifold, and $p\in M$ a given point.
\begin{enumerate}
\item If $det\tmop{CY}_p\neq 0$, or there is a sequence of points $p_k$ converging to $p$ such that $det\tmop{CY}_{p_k}\neq 0$, there are no local LCWs at $p$
\item If there is a neighborhood $U$ of $p$ where $\tmop{CY}\neq 0$ but $\det\tmop{CY}=0$, then $U$ admits a LCW if and only if one of the two one-dimensional distributions of eigenflag directions for $\tmop{CY}$ satisfies one of the equivalent conditions in Theorem \ref{prop: characterization of a conformal product factor}
\item If there is a neighborhood $U$ of $p$ where $\tmop{CY}\vert_U\equiv 0$, then the metric is conformally flat in $U$, and it admits the same LCWs as a subset of $\R^3$.
\end{enumerate}
\end{thm}

The four dimensional case is contained in the following Theorem; the definition of the types of Weyl tensors is deferred to Lemma \ref{lemma: classification of algebraic Weyl tensors}.

\begin{thm}\label{thm:four dimensions}
Let $(M,g)$ be a 4-dimensional manifold and $p\in M$.
\begin{itemize}
 \item If $W$ is of type $A$ at $p$, or there is a sequence of points with Weyl tensors of type $A$ converging to $p$, there are not local LCWs at $p$.
 \item If $W$ is of type B at all points in a neighborhood of $p$, then $M$ admits a local LCW at $p$ if and only if at least one of the four $1$-dimensional distributions defined by eigenflags as in Lemma \ref{lemma: classification of algebraic Weyl tensors} satisfies the hypothesis of Theorem \ref{prop: characterization of a conformal product factor}.
 \item If $W$ is of type C at all points in a neighborhood of $p$, and the two complementary distributions of eigenflag directions satisfy the hypothesis of Theorem \ref{prop: characterization of a conformal product factor}, then $M$ is locally conformal to a product of surfaces, and it admits a local LCW if and only if at least one of the two integral factors is a surface of revolution.
 \item If $W$ is of type D at all points in a neighborhood of $p$, then it is locally conformally flat at $p$, and admits the same LCWs as a subset of $\R^4$.
\end{itemize}
\end{thm}

To write a concrete example of a manifold with an eigenflag vector field, but without any local LCW (showing that the necessary condition from theorem~\ref{Weyl} is not sufficient)
recall that \emph{an ellipsoid} in $\R^3$ is \emph{scalene} if its three axis have different lengths.

\begin{thm}\label{thm:example ellipsoids}
Let $(M,g)=(S_1,g_1)\times (S_2,g_2)$ where $S_1$ and $S_2$ are two scalene ellipsoids with the metric induced by $\mathbb{R}^3$.
Then any open subset of $(M,g)$ satisfies the eigenflag condition but does not admit any LCW.
\end{thm}

The paper starts in Section \ref{section: is this distribution a conformal factor} by considering a metric that becomes a product after a conformal change.
The 3-dimensional case is examined in Section \ref{section: conformal factors in 3 dimensions}. Here
Theorem~\ref{Cotton} says that there are only a finite number of directions than can be factors of a product after conformal change; we analyze them by means of Theorem~\ref{prop: characterization of a conformal product factor}. This will provide a proof of Theorem \ref{thm:three dimensions}.

Section \ref{section: algebraic Weyl tensors} classifies the possible algebraic Weyl tensors that may arise in a $4$-dimensional manifold in terms of its eigenflag directions and eigenspaces.
This allows us to identify all the eigenflag directions.

\begin{lem}\label{lemma: classification of algebraic Weyl tensors}
 The algebraic Weyl operators $W$ in a vector space of dimension $4$ fall into one of the following types:

 \begin{description}
  \item[A] $W$ has no eigenflag directions.
  \item[B] $W$ has at least one eigenflag direction and three different eigenspaces of dimension $2$. In this case, $W$ has exactly four eigenflag directions.
  \item[C] $W$ has at least one eigenflag direction and two eigenspaces with dimensions four and two. In this case, the eigenflag directions for $W$ consist of the union of two orthogonal 2-planes.
  \item[D] $W$ is null. All directions are eigenflag.
 \end{description}
\end{lem}

We use this information in Section \ref{section: conformal factors in 4 dimensions};  as explained above, Theorem~\ref{prop: characterization of a conformal product factor} suffices to deal with case B.
In Section~\ref{6}, we give an interesting example where three of the four eigenflag directions are conformal factors. This is a manifold with three LCWs with orthogonal level sets, that is not conformal to a product of surfaces.

We start  Section \ref{section: conformal factors in 4 dimensions} by proving that product metrics of surfaces have Weyl operators of Types C and D; then we show that if such a product of surfaces admits a LCW and its Weyl operator is of type C (not the trivial conformally flat case), we can choose coordinates where the metric assumes a specially nice form. As a result, we are able to prove

\begin{thm}\label{thm: products of surfaces which admit LCW}

  Let $(S_ 1, g_ 1)$ and $(S_2,g_2)$ be open subsets of $\R^2$ with Riemannian metrics. Assume that the Weyl operator of the product metric does not vanish at any point.
  The following are equivalent:
  \begin{itemize}
    \item $(S_1, g_1)$ is locally isometric to a surface of revolution
    
    \item $(S_1, g_1)$ has a non-trivial Killing vector field
    
    \item $(S_1	\times S_2, g_1\times g_2)$  admits a LCW that is everywhere tangent to the first
    factor
    
  \end{itemize}
\end{thm}

Hence Theorem \ref{thm:example ellipsoids} is a corollary of this,
as scalene ellipsoids are not surfaces of revolution and their product satisfy the condition on the Weyl operator. 

We are left with the case where  a manifold may have a Weyl tensor of type C, but not be conformal to a product of surfaces. Therefore in principle there are many candidates to be one
dimensional conformal factors. In section \ref{section: example FH}, we show that this is indeed possible, and show to deal with this situation in an specific example and explain how to proceed in general.

Finally, we would like to point out that, in principle, a similar analysis might be conducted in higher dimensions though satisfying the eigenflag condition is more rare (see  \cite[Theorem 6.1]{AFGR} for a quantitative statement in this regard), and the analysis is bound to become much more complicated.

\section{Criteria for a conformal product}\label{section: is this distribution a conformal factor}
Suppose we are given a Riemannian manifold $(M,g)$ with a vector distribution $D\subset TM$ such that both $D$ and its orthogonal complement $D^\perp$ are integrable. The main result of this section establishes criteria that will assure that $M$ is locally conformal to a product metric with $D$ and $D^\perp$ tangent to the factors. In order to state them, we will need some notation.

    The orthogonal splitting $TM=D\oplus D^\perp$ induces bundle projections 
    \[
    P_D:TM\to D, \quad P_{D^\perp}:TM\to D^\perp;
    \]
we will alternatively denote by $X^D$ and $X^\perp$ to the components of a vector field under the above splitting.  
    
    Given a metric tensor $g$ in $M$, we denote by $g^D$ and $g^{\perp}$ the restrictions of $g$ to $D$ and $D^\perp$ respectively, that is 
    \[
    g^D(X,Y):=g(P_D(X),P_D(y)), \quad g^{\perp}(X,Y):=g(P_{D^\perp}(X),P_{D^\perp}(y)).
    \]
It is clear that 
\[
g(X,Y)=g^D(X,Y)+g^{\perp}(X,Y).
\]
for any pair of vector fields $X$, $Y$.
   
\begin{lem}\label{lem:joint integrability}
Assume $D$, $D^\perp$ are integrable distributions of dimensions $d$ and $n-d$ respectively. Then for any point $p\in M$, there is a coordinate chart $(U,\phi)$ with $\phi=(x_1,\dots,x_n)$ such that 
\begin{enumerate}
\item $\phi(p)=0$, $\phi(U)=(-1,1)^n$;
\item the integral manifolds for $D$ in $U$ are given by equations 
\[
(x^{d+1},\dots, x^n)=(a_{d+1},\dots,a_n);
\]
\item the integral manifolds for $D^\perp$ in $U$ are given by equations 
\[
(x^{1},\dots, x^d)=(a_{1},\dots,a_d).
\]
\end{enumerate}
\end{lem}    
\begin{proof}
Frobenius theorem gives us charts $(U_1,\phi_1)$ and $(U_2,\phi_2)$, with 
$\phi_1=(y_1,\dots,y_n)$, $\phi_2=(z_1,\dots,z_n)$, $\phi_1(p)=\phi_2(p)=0$, and
 such that the integral submanifolds for $D$ and $D^\perp$ are given respectively by the equations $(y_{d+1}, \dots, y_n)=
(b_{d+1},\dots,b_n)$ and  $(z_{1}, \dots, z_d)=
(c_{d+1},\dots,c_n)$ for constants $b_i$, $c_j$. The map 
$\phi=(z_{1}, \dots, z_d,y_{d+1}, \dots, y_n)$, defined in a neighborhood $U$ of $p$, is a local diffeomorphism at $p$ with $\phi(p)=0$, thus it defines a coordinate chart in some neighborhood of $p$. 
Taking a smaller neighborhood if needed, and with the help of a linear change of coordinates in $\R^n$, we can assume that its image is $(-1,1)^n$ as required. 
\end{proof}    
    
It is clear from the Lemma that, if $N$ and $N^\perp$ are the integral manifolds for $D$ and $D^\perp$ through $p$ in $U$, then $U$ is diffeomorphic to $N\times N^\perp$. The aim of this section is to find conditions on a metric $g$ in $U$ such that $(\phi^{-1})^*g$ is conformal to a product metric on $N\times N^\perp$. 
 In order to do this, we need to introduce the following $1$--form $\Phi$ in $M$:
\begin{equation}
\Phi (X) := \Tr_g\left(\frac{1}{d} \mathcal{L}_{X^\perp} g^D + \frac{1}{n - d} \mathcal{L}_{X^D} g^{\perp}\right)
\end{equation} 
  Although the Lie derivative of some tensor $T$, $X \rightarrow \mathcal{L}_X T$, is not tensorial
  in $X$, we have that 
\begin{multline}
\Lie_{fX^\perp}g^D(Y,Z)=f\Lie_{X^\perp}g^D(Y,Z)-g^D(Yf X^\perp, Z)-g^D(Y, Zf X^\perp)=\\ = f\Lie_{X^\perp}g^D(Y,Z)
\end{multline}
and similarly for $\Lie_{X^D}g^\perp$, so $\Phi$ is effectively a 1-form. 

\begin{defn}\label{def: umbilic distribution}
A distribution $D$ is said to be \emph{umbilical} if there exists a vector field $H\in D^\perp$, called the \emph{mean curvature vector field of $D$}, such that for $X,Y \in D$ and $Z \in D^{\perp}$ it holds that
\begin{equation}
\langle \nabla_X Y, Z \rangle =  \langle X, Y \rangle \langle Z, H
    \rangle
\end{equation}
\end{defn}
The relation of umbilicity to LCWs was already noted in \cite{DKSU07}.

\begin{rem}
If $D$ is an arbitrary smooth distribution, then we can define the second fundamental form of the distribution
$$
II(X,Y) = P_{D^\perp}(\nabla_X Y)
$$
where $P_{D^\perp}$ is the projection onto $D^\perp$ and $X$, $Y$, are vector fields tangent to $D$.
It is easy to check that the distribution is integrable if and only this form is symmetric, and in particular, an umbilical distribution is integrable.
\end{rem}

\begin{thm}\label{prop: characterization of a conformal product factor}
 Let $(M, g)$ be a Riemannian metric, and $D$, $D^\perp$  integrable distributions as above. 

 The following are equivalent:
 \begin{enumerate}[label=(\arabic*),ref=(\arabic*)]
 \item $g$ is locally conformal to the product of the metric restricted to an integral leave of $D$ and the metric restricted to an integral leave of $D^\perp$; \label{uno}
 \item the Lie derivative of $g^D$ with respect to any vector field in $D^\perp$ is a multiple of $g^D$, the Lie derivative of $g^\perp$ with respect to any vector field in $D$ is a multiple of $g^\perp$, and the 1-form $\Phi$ is closed; \label{dos}
 \item the distributions $D$ and $D^\perp$ are umbilic and, if $H_1$ and $H_2$ are the respective 
mean curvature vector fields, $H_1 + H_2$ is a {\emph{gradient vector field}} for the metric $g$. \label{tres}
 \end{enumerate}
\end{thm}

\begin{proof}
\Implies{uno}{dos} Let $\tilde{g} = e^{\alpha} g$ be a product
  metric of two metrics on the leaves of $D$ and $D^\perp$. Using the chart $(U,\phi)$, there are symmetric tensors in $U$ such that  $\tilde{g}$ is
  written as
  \[ \tilde{g} = \left(\begin{array}{cc}
       e^{\alpha} g_1 & \\
       & e^{\alpha} g_2
     \end{array}\right) \]
  where the first block does not depend on the coordinates of the second factor
  and vice versa. This is equivalent to
  \[ 
  \mathcal{L}_{X^\perp} (e^{\alpha} g_1) =\mathcal{L}_{X^D} (e^{\alpha} g_2) =
     0, 
     \]
  and also to
  \begin{eqnarray*}
    \mathcal{L}_{X^\perp} (g_1) & = & - X^\perp (\alpha) g_1\\
    \mathcal{L}_{X^D} (g_2) & = & - X^D (\alpha) g_2
  \end{eqnarray*}

  Thus the first part of \ref{dos} is immediate, and the second follows because 
  \[
  \Phi (X) = - X (\alpha) = -
  d \alpha (X).
  \] 
  
\Implies{dos}{uno} $D$ is integrable since, for $X, Y \in D$ and
  $Z \in D^\perp$:
  \begin{eqnarray*}
    g^\perp ([X, Y], Z) & = & -(\mathcal{L}_X g^\perp)(Y, Z)+ X (g^\perp (Y, Z))  - g^\perp(Y, [X, Z])\\
    & = & - c g^\perp (Y, Z)\\
    & = & 0
  \end{eqnarray*}
  where we are using that $\mathcal{L}_X g^\perp= c g^\perp$ for some function $c$. 
	The integrability of $D^\perp$ is proved similarly.  
  
  $U$ is simply connected, thus $\Phi$ being closed yields some function $\alpha$ such that $\Phi = - d\alpha$. For an arbitrary vector field $X$, the hypothesis in part \ref{dos} gives some function $c$ such that 
  $\Lie_{X^\perp}g^D=c g^D$; taking traces in both sides, and using that the dimension of $D$ is $d$, we get that 
  \[
  c= \frac{1}{d}\Tr_g\left(\Lie_{X^\perp}g^D\right)=\Phi(X^\perp)=- d\alpha(X^\perp)=-X^\perp(\alpha), 
  \]
  thus $\Lie_{X^\perp}g^D+X^\perp(\alpha)=0$, and a simple computation yields $\Lie_{X^\perp}(e^\alpha g^D)=0$. A similar procedure gives us $\Lie_{X^D}(e^\alpha g^\perp)=0$, so the metric $\tilde{g} =
  e^{\alpha} g$ is a product of two metrics on the leaves of $D$ and $D^\perp$.

\Implies{uno}{tres}   Let $\tilde{g} = e^{\alpha} g$ be a product of two metrics on the leaves of $D$ and $D^\perp$ as before. A simple
  formula relates the Levi-Civita connections of $g$ and $\tilde{g}$:
  \begin{equation}
    \nabla_X Y = \tilde{\nabla}_X Y + g(X, Y) U - g(X, U) Y - g(Y, U)X \label{connection of rescaled metric}
  \end{equation}
  where $U = - \nabla \alpha$. Thus for $X, Y \in D$ and $Z, W
  \in D^\perp$,
  \[ \begin{array}{lll}
       g(\nabla_X Y, Z) & = & g(X, Y)g(U, Z)
     \end{array} \]
  so the projection of $U$ onto $D^\perp$ is the mean curvature of $D$ and
  vice versa. It follows that $H_1 + H_2 = U$, which is a gradient.
  
  \Implies{tres}{uno} Suppose $H_1 + H_2 = - \nabla \alpha$, and
  define $\tilde{g} = e^{- \alpha} g$. Equation (\ref{connection of
  rescaled metric}) shows that
  \[ \begin{array}{ll}
       g(\tilde{\nabla}_X Y, Z) & =
     \end{array} 0 \]
  \[ \begin{array}{ll}
      g(\tilde{\nabla}_Z W, X) & =
     \end{array} 0 \]
  for $X, Y \in D$ and $Z, W \in D^\perp$.
  
  It follows from equation \eqref{def: umbilic distribution} that umbilic distributions
  are integrable, so by Lemma \ref{lem:joint integrability}, we can find a coordinate basis adapted simultaneously to $D$ and
  $D^\perp$, i.e. $\{ \partial_1, \ldots, \partial_d \}$ span $D$ and $\{ \partial_{d
  + 1}, \ldots, \partial_n \}$ span $D^\perp$. Then for $i \in \{ 1 \ldots d \}$,
  $j, k \in \{ d + 1 \ldots n \}$:
  \begin{eqnarray*}
    \partial_i \tilde{g}_{jk} & = & \partial_i (\tilde{g} (\partial_j,
    \partial_k))\\
    & = & \tilde{g} (\tilde{\nabla}_{\partial_i} \partial_j, \partial_k) +
    \tilde{g} (\partial_j, \tilde{\nabla}_{\partial_i} \partial_k)\\
    & = & - \tilde{g} (\tilde{\nabla}_{\partial_j} \partial_i, \partial_k) -
    \tilde{g} (\partial_j, \tilde{\nabla}_{\partial_k} \partial_i) - \tilde{g}
    ([\partial_i, \partial_j], \partial_k) - \tilde{g} ([\partial_i,
    \partial_k], \partial_j)\\
    & = & - \partial_j (\tilde{g} (\partial_i, \partial_k)) + \tilde{g}
    (\partial_i, \tilde{\nabla}_{\partial_j} \partial_k) - \partial_k
    (\tilde{g} (\partial_i, \partial_j)) + \tilde{g} (\partial_i,
    \tilde{\nabla}_{\partial_k} \partial_j)\\
    & = & 0
  \end{eqnarray*}
  so $\tilde{g}$ is a product metric.
\end{proof}
Related conditions can be found in the literature.
For example, \cite{Tojeiro} investigates the case of warped products. However we have preferred to keep our criteria as simple as possible. 

\section{Conformal factors in  dimension three}\label{section: conformal factors in 3 dimensions}

This section combines the results of the previous section with an examination of the Cotton-York tensor of a metric $g$ to describe when a 3-dimensional manifold has a LCW.
We refer the reader to \cite{AFGR} for some background on the Cotton-York tensor.
Recall that the space of algebraic Cotton-York tensors at some given point $p\in M$ coincides with the space of traceless symmetric operators in $T_pM$. 

\begin{defn}
 An \strong{eigenflag direction} of a traceless symmetric operator in a three-dimensional euclidean space $V$ is a one-dimensional vector subspace $L$ such that for any $v \in L$ and $w_1, w_2 \in L^{\perp}$, we have
\[ CY_p (v, v) = CY_p (w_1, w_2) = 0 \]
\end{defn}

Suppose we are given a metric $g$ in $M$; Theorem 1.6 in \cite{AFGR} shows that if a conformal metric $\tilde{g}$ admits a parallel vector field, the subspace $L$ that generates is an eigenflag direction of the Cotton-York tensor of $g$ at each point of $M$.

We start by classifying the possible algebraic Cotton-York tensors in terms of their eigenflag directions.

\begin{lem}\label{lem:algebraic Cotton}
  An algebraic Cotton-York tensor $\tmop{CY}$ falls into one of the following categories:
  \begin{itemize}
    \item Every direction in $V$ is eigenflag for $\tmop{CY}$; this agrees with the case when $\tmop{CY}$ is null.
    
    \item There are exactly two eigenflag directions for $\tmop{CY}$; this agrees with the case when $\tmop{CY}$ is not null and $\det (\tmop{CY}) = 0$.
    
    \item There are not eigenflag directions; this is equivalent to the case $\det (\tmop{CY}) \neq 0$.
  \end{itemize}
\end{lem}

\begin{proof}
Lemma 5.1 in \cite{AFGR} shows that $\det (\tmop{CY}) = 0$ if and only if $\tmop{CY}$ has an eigenflag direction, so assume that $L$ is an eigenflag direction for $\tmop{CY}$. If $v$ is a unit vector along $L$, and $\{ v, w_1, w_2\}$ is an orthonormal basis of $V$, the matrix of $\tmop{CY}$ is given by
\[ \left(\begin{array}{ccc}
     0 & a & b\\
     a & 0 & 0\\
     b & 0 & 0
   \end{array}\right) \]
   (see Lemma 1.7 in \cite{AFGR}).
   
A further rotation with axis $L$ changes the matrix of $\tmop{CY}$ to the form
\[ \left(\begin{array}{ccc}
     0 & c & 0\\
     c & 0 & 0\\
     0 & 0 & 0
   \end{array}\right) \]
If $\tmop{CY}$ is not zero, the null directions for this symmetric operator is the union of the planes $\{x = 0 \}$ and $\{ y = 0 \}$.
Thus only the orthogonals to those two planes can be eigenflag directions.

\end{proof}

\begin{proof}[Proof of Theorem \ref{thm:three dimensions}]
In the first case, the metric is conformally flat; in the second it follows from Lemma \ref{lem:algebraic Cotton} and Theorem \ref{prop: characterization of a conformal product factor}.
\end{proof}

\section{Classification of  algebraic Weyl tensors of $4$-manifolds}\label{section: algebraic Weyl tensors}

We would like to carry out a similar analysis for 4-dimensional manifolds to the one we did in Section \ref{section: conformal factors in 3 dimensions}; however, the extra dimension brings out a higher complexity,  even at the level of algebraic Weyl operators. So in this section, we will start by classifying such operators with respect to the size of their set of eigenflag directions.

\begin{proof}[Proof of Lemma \ref{lemma: classification of algebraic Weyl tensors}]
Let $W$ be a Weyl operator with an eigenflag direction $L = \langle v \rangle$.
The operator $W|_{\langle v \wedge v^{\perp} \rangle}$ is symmetric, and
diagonalizes in an orthonormal basis $v \wedge e_2, v \wedge e_3, v \wedge
e_4$. Define $e_1 = v$, and $e_{ij} = e_i \wedge e_j$. It follows from the
properties of the Weyl operator that $e_{23}$, $e_{24}$ and $e_{34}$ are eigenvectors  of $W$ with the
same eigenvalues as $e_{14}$, $e_{13}$, $e_{12}$ respectively (see the proof of Theorem 6.1 in \cite{AFGR}).
Thus $W$ diagonalizes as

\begin{equation}\label{diagonal Weyl operator}
W=\left(\begin{array}{cccccc}
  \lambda_{12} &  &  &  &  & \\
  & \lambda_{13} &  &  &  & \\
  &  & \lambda_{14} &  &  & \\
  &  &  & \lambda_{12} &  & \\
  &  &  &  & \lambda_{13} & \\
  &  &  &  &  & \lambda_{14}
\end{array}\right) 
\end{equation}

Recall that $W$ is traceless, so $\lambda_{12} + \lambda_{13} + \lambda_{14} = 0$.
If the three numbers $\lambda_{12}, \lambda_{13}, \lambda_{14}$ are different, the operator has exactly three eigenspaces, each of dimension $2$.
If two numbers coincide, there is one eigenspace of dimension $4$ and a second eigenspace of dimension $2$.
Finally, the remaining possibility is that the Weyl operator vanishes. This will account for the different possibilities in the statement of the Theorem, once we have related them to the eigenflags. We will do this case by case.

\textbf{Three different eigenvalues.}
Suppose $\lambda_{12}\neq\lambda_{13}\neq\lambda_{14}$; then
the eigenspace for $\lambda_{12}$ is the set of bivectors of the form $ae_{12} + be_{34}$, $a,b\in \R$.
The Plücker relations imply that such bivector is simple only when $ab = 0$.
In other words, the only simple bivectors in the eigenspace for $\lambda_{12}$ are the multiples of either $e_{12}$ or of $e_{34}$.
Changing $i,j$ we get that every simple eigenvector of $W$ is a multiple of some $e_{ij}$, and consequently each one of the $e_i$'s is an eigenflag direction.

Suppose $v$ were a vector spanning a different eigenflag direction.
Then $v\wedge v^\perp$ would be an eigenspace of $W$ consisting of simple bivectors, so there should be three orthogonal unit vectors $w_k \in v^{\perp}$ such that $v \wedge w_k$ are eigenvectors for $W$.
It follows that $v \wedge w_1 = e_{ij}$ for some $i, j$, which implies that $v, w_1 \in \tmop{span} (e_i, e_j)$, and we can assume without lost of generality that $i = 1, j = 2$.
Then $w_2 \in (v, w_1)^{\perp} = \tmop{span} (e_3, e_4)$, and $v \wedge w_2$ can only be a coordinate $2$-plane if $v = e_1$ or $v = e_2$.

\textbf{Two different eigenvalues.}
Suppose $\lambda_{12} = \lambda_{13}\neq \lambda_{14}$. Let $v = ae_2 + be_3$ be any vector
in $\tmop{span} (e_2, e_3)$. Then 
\[
v\wedge e_1=-ae_{12}- be_{13}, \quad v\wedge e_4=ae_{24} + be_{34},
\]
are eigenvectors of eigenvalue $\lambda_{12}$, while 
\[v\wedge(- be_2 + ae_3) = (ae_2 + be_3) \wedge
(- be_2 + ae_3) = (a^2 + b^2) e_{23}
\]
is an eigenbivector of eigenvalue $\lambda_{14}$. Therefore $v \wedge v^{\perp}$ is an invariant subspace,
and $v$ is an eigenflag direction. A similar argument applies to any vector in
$\tmop{span} (e_1, e_4)$ to show that it is an eigenflag direction, hence we only need to prove that there are no additional eigenflag directions. 

A general bivector in the $\lambda_{12}$-eigenspace has the form $ae_{12} +
be_{34} + ce_{13} + de_{24}$, and it is simple when
\[ ab - cd = 0 \]
This equation defines a $3$-dimensional quadric in a $4$ dimensional space,
and does not contain any linear space of dimension $3$. If $v$ is an eigenflag
direction, $v \wedge v^{\perp}$ is an invariant subspace, and the restriction
of $W$ to $v \wedge v^{\perp}$ diagonalizes in subspaces of the eigenspaces
for $W$. Thus, $v \wedge v^{\perp}$ must intersect the eigenspace associated
to $\lambda_{14}$. The intersection is spanned by a bivector $v \wedge w$, for
some $w \in v^{\perp}$, but the only simple eigenbivectors in the eigenspace
associated to $\lambda_{14}$ are $e_{14}$ and $e_{23}$. If $v \wedge w =
e_{14}$, this implies that $v \in \tmop{span} (e_1, e_4)$, while $v \wedge w =
e_{23}$ implies $v \in \tmop{span} (e_2, e_3)$.
\end{proof}

\section{Conformal factors in dimension four}\label{section: conformal factors in 4 dimensions}

A 4-dimensional manifold may be conformal to a product in two different ways: a product $\R\times N$, with $N$ a 3-dimensional manifold, or as a product of two surfaces.
Since we are going to use the Weyl tensor to distinguish between the two cases, we will start by computing the Weyl operator of a product of two surfaces.

\begin{lem}\label{lemma: Weyl operator of a product of surfaces}
 Let $M_1$ and $M_2$ be two Riemannian surfaces.
 \begin{enumerate}
\item The Weyl operator of $M_1\times M_2$ has type $C$ or $D$ at any point.
 
\item  It has type $D$ at every point if and only if both surfaces have constant curvature $s_1$ and $s_2$, and $s_1+s_2=0$.

\item If it has type $C$ at a point $p$, then the two planes of eigenflag directions of $W_p$ of Lemma \ref{lemma: classification of algebraic Weyl tensors} are the distributions tangent to the factors $M_1$ and $M_2$ at $p$.
\end{enumerate}
\end{lem}
\begin{proof}
Taking isothermal coordinates $(t, x)$ on $M_1$ and $(y, z)$ on $M_2$,
the product metric is written as

$$g = \left(\begin{array}{cccc}
  f (t, x) &  &  & \\
  & f (t, x) &  & \\
  &  & h (y, z) & \\
  &  &  & h (y, z)
\end{array}\right).$$

Denote the normalized coordinate fields as
\[
\hat{\partial_t}=
\frac{1}{\sqrt{f}}\partial_t, \quad \hat{\partial_x}=\frac{1}{\sqrt{f}}\partial_x,\quad \hat{\partial_y}=\frac{1}{\sqrt{h}}\partial_y, \quad \hat{\partial_z}=\frac{1}{\sqrt{h}}\partial_z.
\]
In the basis of $\Lambda^2T_pM$ 
\[
\hat{\partial}_t \wedge \hat{\partial}_x, 
\quad \hat{\partial}_t \wedge \hat{\partial}_y, 
\quad
\hat{\partial}_t \wedge \hat{\partial}_z,
\quad 
\hat{\partial}_y \wedge \hat{\partial}_z,
\quad 
\hat{\partial}_x \wedge \hat{\partial}_z, 
\quad
\hat{\partial}_x \wedge \hat{\partial}_y ,
\]
the Weyl operator is diagonal, with
\begin{equation}
\begin{array}{c}
     \lambda \xequal{\tmop{def}} \langle W (\hat{\partial}_t \wedge \hat{\partial}_x)
     \nocomma \nocomma, \hat{\partial}_t \wedge \hat{\partial}_x \rangle = \langle W
     (\hat{\partial}_y \wedge \hat{\partial}_z) \nocomma \nocomma, \hat{\partial}_y \wedge
     \hat{\partial}_z \rangle =\\
     \\
     \dfrac{h_y^2 + h_z^2 - hh_{yy} - hh_{zz}}{6 h^3} + \dfrac{f_t^2 + f_x^2 -
     ff_{tt} - ff_{xx}}{6 f^3}
   \end{array} 
\end{equation}
and
\begin{equation}
 \begin{array}{c}
     \langle W (\hat{\partial}_t \wedge \hat{\partial}_y) \nocomma \nocomma, \hat{\partial}_t
     \wedge \hat{\partial}_y \rangle = \langle W (\hat{\partial}_t \wedge \hat{\partial}_z)
     \nocomma \nocomma, \hat{\partial}_t \wedge \hat{\partial}_z \rangle =\\
     \\
     \langle W (\hat{\partial}_x \wedge \hat{\partial}_y) \nocomma \nocomma, \hat{\partial}_x
     \wedge \hat{\partial}_y \rangle = \langle W (\hat{\partial}_x \wedge \hat{\partial}_z)
     \nocomma \nocomma, \hat{\partial}_x \wedge \hat{\partial}_z \rangle =\\
     \\
     - \dfrac{h_y^2 + h_z^2 - hh_{yy} - hh_{zz}}{12 h^3} - \dfrac{f_t^2 + f_x^2
     - ff_{tt} - ff_{xx}}{12 f^3} = - \lambda / 2
   \end{array}
\end{equation}
Thus $W$ falls into type $C$ or $D$ depending on whether $\lambda$ is different or equal to zero. 

We recognize 
\[
\lambda(t,x,y,z) = \frac{-1}{3}\left(s_1(t,x) + s_2(y,z)\right),
\]
where $s_i$ is the Gaussian curvature of $M_i$.
Thus if $\lambda$ vanishes identically, $s_1$ and $s_2$ are global constants whose sum is $0$.

\end{proof}

We are now ready to distinguish those metric products of surfaces that admit, simultaneously, a LCW.

\begin{lem}\label{lem:chart product of surfaces}
 Let $(M,g)$ be a product of surfaces that admits a limiting Carleman weight $\varphi$. Let $p\in M$ with $W_p\neq 0$; then there are coordinates $(t,x,y,z)$ around $p$ in which $g$ is written as
\begin{equation}
 g = \left(\begin{array}{cccc}
  e^{J (x)} &  &  & \\
  & e^{J (x)} &  & \\
  &  & e^{K (y, z)} & \\
  &  &  & e^{K (y, z)}
\end{array}\right)
\end{equation}
for some functions $J=J(x)$, $K=K(y,z)$.
\end{lem}
\begin{proof} Since $W_p\neq 0$, Lemma \ref{lemma: Weyl operator of a product of surfaces} implies that $W$ is of type $C$ in some neighborhood of $p$, and the set of eigenflag directions for the Weyl tensor of $g$ and $\tilde{g}$ is the union of the two planes $\langle \partial_t, \partial_x \rangle$ and $\langle \partial_y, \partial_z \rangle$.

Choose some coordinate system $(t,x,y,z)$ adapted to the product structure $M_1\times M_2$, as in the proof of Lemma \ref{lemma: Weyl operator of a product of surfaces}.

Let $\tilde{g}$ be the rescaled metric
\[ \tilde{g} = | \nabla \varphi |^2 g.
\]
Lemmas 3.10 and 3.11 in \cite{SaloLN} give that the vector field $A = \tilde{\nabla} \varphi$ is a parallel vector field for $\tilde{g}$. 
Theorem 1.3 in \cite{AFGR} shows that $\langle A\rangle$ must be an eigenflag direction.
 
Without loss of generality, we will assume that $A$ is in $\langle\partial_t, \partial_x\rangle$ at every point, hence there are functions $\alpha,\beta$ such that 
\[
A=\alpha\partial_t+ \beta\partial_x.
\]
 
  It follows from the definition of parallel vector field that $\Delta$, the distribution orthogonal to $A$, is integrable (see the proof of lemma 3.12 in page 57 of \cite{SaloLN}).
It is clear that 
\[
\Delta=\left<\,-\beta\partial_t+ \alpha\partial_x, \partial_y, \partial_z\,\right>.
\]
Denote $B:=-\beta\partial_t+ \alpha\partial_x$; the integrability condition implies that 
\[
\left[B,\partial_y\right]=(\partial_y \beta)\partial_t
-(\partial_y \alpha)\partial_x\in \Delta
\]
and therefore $\left[B,\partial_y\right]$ must be in the $B$ direction, i.e, $\left[B,\partial_y\right]=\lambda_1 B$ for some function $\lambda_1$; thus
\[
\partial_y \beta = -\lambda_1\cdot \beta, \quad \partial_y \alpha=-\lambda_1\cdot \alpha
\]
Looking at $\left[B,\partial_z\right]$, we get 
\[
\partial_z \beta = -\lambda_2\cdot \beta, \quad \partial_z \alpha=-\lambda_2\cdot \alpha
\]

Next, observe that any integral submanifold of $\Delta$ is foliated by surfaces tangent to $\langle \partial_y, \partial_z \rangle$; since $A$ is a parallel vector field, its integral curves are given by geodesics. Therefore, if we consider the maps
\[
\phi_{(t_0,x_0)}(u,v,w)=\exp_{(t_0,x_0,v,w)}\left(u A\right),
\]
we see that for each $(t_0,x_0)$, the vector fields $\partial_u$, $\partial_v$, $\partial_w$ are mapped by the differential of $\phi_{(t_0,x_0)}$ to $A$, $\partial_y$ and $\partial_z$, and therefore
\[
\left[A,\partial_y\right]=\left[A,\partial_z\right]=0.
\]
On the other hand, writing $A=\alpha\partial_t+ \beta\partial_x$, we would get
\[
\left[A,\partial_y\right]=-(\partial_y\alpha)\partial_t
-(\partial_y \beta)\partial_x, \quad 
\left[A,\partial_z\right]=-(\partial_z\alpha)\partial_t
-(\partial_z \beta)\partial_x.
\]
The consequence of this is that $\lambda_1=\lambda_2=0$, and every Lie bracket between elements in the basis $\{ A, B, \partial_y, \partial_z \}$ vanishes, and hence they form a coordinate basis for some set of coordinates about $p$.
We will keep on denoting them $(t,x,y,z)$, although only the last two would coincide with the former.
The first two, $(t,x)$, will still, however, parametrize the first factor of the product structure of $M$.
In this chart, $\tilde{g}$ is written as
\begin{equation}
 \tilde{g} = \left(\begin{array}{cccc}
  1 &  &  & \\
  & 1 &  & \\
  &  & b (t, x, y, z) & \\
  &  &  & b (t, x, y, z)
\end{array}\right)
\end{equation}
where $b$ is actually $h (y, z) | \nabla \varphi |^2 (t, x, y, z)$.

Because $A$ is parallel, it follows from Koszul formula
\begin{multline}
  2 g (\nabla_X Y, Z)  = \partial_X (g (Y, Z)) + \partial_Y
  (g (X, Z)) - \partial_Z (g (X, Y))\\
   + g ([X, Y], Z) - g ([X, Z], Y) - g ([Y, Z],
  X),
\end{multline}
that $b$ does not depend on $t$ (taking $X = A$, $Y = Z = \partial_y$).

The metric $\tilde{g}$ is conformal to a product where the factors are tangent to $\langle \partial_t, \partial_x \rangle> = \langle A, B \rangle$ and $\langle \partial_y, \partial_z \rangle$.
Let $\hat{\partial}_y =\frac{1}{\sqrt{b}} \partial_y$ be a unit vector in the direction of $\partial_y$.
We compute the Weyl operator in the basis 
\[
A \wedge B, \quad A \wedge \hat{\partial}_y,\quad  A \wedge \hat{\partial}_z,\quad  \hat{\partial}_y \wedge \hat{\partial}_z, \quad  B \wedge \hat{\partial}_z,\quad  B \wedge \hat{\partial}_y .
\]
The $6\times 6$ matrix of $W$ in this basis has the block structure
\[
W = \left(\begin{array}{cc}
  W_1 & 0\\
  0 & W_1
\end{array}\right)
\]
where the $3\times 3$ block $W_1$ is
\[ 
\begin{pmatrix}
     - \Psi & \frac{bb_{xy} - b_x b_y}{4 \,
     b^{5/2}} & \frac{bb_{xz} - b_x b_z}{4 \,
     b^{5/2}}\\
     \frac{bb_{xy} - b_x b_y}{4 \, b^{5/2}} &
     \Psi/2 & 0\\
     \frac{bb_{xz} - b_x b_z}{4 \, b^{5/2}} & 0 &
     \Psi/2
\end{pmatrix}
\]
where $\Psi = \frac{bb_x^2 - b^2 b_{xx} - b_y^2 + bb_{yy} - b_z^2 + bb_{zz}}{6\, b^3}$.

As mentioned in the proof of Lemma \ref{lemma: Weyl operator of a product of surfaces}, if the metric is a product of surfaces tangent to $\langle \partial_1, \partial_2 \rangle$ and $\langle \partial_3, \partial_4 \rangle$, the off-diagonal elements in the corresponding basis of $\Lambda^2(T_p M)$ must vanish at every $p$.
Thus:
\begin{eqnarray*}
  bb_{xy} & = & b_x b_y\\
  bb_{xz} & = & b_x b_z
\end{eqnarray*}
Integrating the first equation we get:
\begin{eqnarray*}
  \frac{(b_x)_y}{b_x} & = & \frac{b_y}{b}\\
  \log (b_x) & = & \log (b) + C (x, z)\\
  \frac{b_x}{b} & = & e^{C (x, z)}\\
  \log (b) & = & D (x, z) + E (y, z)
\end{eqnarray*}
and similarly, integrating the second, we get $\log (b) = F (x, y) + G (y, z)$.
Define $r = \log (b)$, and observe that $r_{xy} = r_{xz} = 0$, which in turn
yields $r_x = H (x)$, or $r (x, y, z) = -J (x) + K (y, z)$.

Thus, in the basis $\{ A, B, \partial_y, \partial_z \}$, the metric $\tilde{g}$ is written  as

\[
\tilde{g} = \left(\begin{array}{cccc}
  1 &  &  & \\
  & 1 &  & \\
  &  & e^{-J (x)} e^{K (y, z)} & \\
  &  &  & e^{-J (x)} e^{K (y, z)}
\end{array}\right)
\]
Undoing the change from $g$ to $\tilde{g}$ shows that the matrix for $g$ is as claimed in the statement of the Theorem.
\end{proof}

We can now proceed with the proof of Theorem \ref{thm: products of surfaces which admit LCW}

\begin{proof}[Proof of Theorem \ref{thm: products of surfaces which admit LCW}]
Assume $S_1\times S_2$ has a LCW tangent to the first factor; Lemma \ref{lem:chart product of surfaces} shows that there are coordinates in $S_1$ such that $g_1$
can be written in the form:
\begin{equation}\label{metric J(x)} 
g = \left(\begin{array}{cc}
  e^{J (x)} & \\
  & e^{J (x)}
\end{array}\right) 
\end{equation}
Clearly such a metric has $\partial_t$ as a Killing field.

Assume now $(S_1,g_1)$ has a Killing field. Since
in dimension two, any Killing field can be completed to a coordinate chart, we can assume that there are coordinates where $g_1$ is written as before. 
It is then clear that after a change of coordinates of the form $(t, x) \rightarrow (t, \lambda (x))$, with $\lambda' (x) = e^{J (x) / 2}$, the matrix of $g$ becomes
\[g = \left(\begin{array}{cc}
  e^{- J (x)} & \\
  & 1
\end{array}\right)\]
which is a piece of a surface of revolution.

Finally, if $(S_1,g_1)$ is a surface of revolution, take a chart as above; undoing the change of coordinates $(t, x) \rightarrow (t, \lambda (x))$ get a chart in the product where the 
 metric $g_1\times g_2$ takes the expression appearing in Lemma \ref{lem:chart product of surfaces}. Multiplying by the conformal factor $e^{-J(x)}$ shows that there is a LCW along $\partial_t$.

\end{proof}

\begin{proof}[Proof of Theorem \ref{thm:four dimensions}]
The main result of \cite{AFGR} shows that if $W_p$ has type A, there cannot be LCWs around $p$. Clearly, the same thing happens if $p$ is in the closure of the set of points with Type A Weyl operators.

If $p$ has a neighborhood where every point has a Type C Weyl operator, Lemma \ref{lemma: classification of algebraic Weyl tensors} shows that there are two orthogonal distributions $D$ and $D^\perp$; the hypothesis of the Theorem assure that there is a conformal change such that a neighborhood of $p$ is isometric to a product of surfaces; since one of them has a Killing field, there is a LCW.

The case of a Type B neighborhood is similarly done. Finally, if $W$ vanishes in a neighborhood of $p$, the manifold would be conformally flat around $p$. 
\end{proof}

\section{A metric with Weyl tensor of constant type B, three of the four eigenflag directions are conformal factors}\label{6}

Lemma \ref{lemma: classification of algebraic Weyl tensors} has an interesting consequence: if a manifold admits a LCW, then it has at least four vector fields of eigenflag directions.
In general, only one of them will really correspond to a conformal factor.
Theorem~\ref{example 3 out of 4 eigenflag are conformal factors} below shows a less usual case of a manifold with Weyl tensor of constant type B that is conformal to a product along three of the four eigenflag directions.

\begin{thm}\label{example 3 out of 4 eigenflag are conformal factors}
  Take coordinates $(t,x,y,z)$ on an open set $$U\subset\{(t,x,y,z)\in\R^4:x>0\}.$$
 Define the metric 
\begin{equation}\label{metric 3 out of 4 eigenflag are conformal factors}
g = \left(\begin{array}{cccc}
1 & 0 & 0 & 0 \\
0 & 1 & 0 & 0 \\
0 & 0 & x & 0 \\
0 & 0 & 0 & x^2
\end{array}\right)
\end{equation} 
 on $U$. Then:
 \begin{itemize}
  \item The Weyl tensor of $g$ has type B at every point.
  \item The eigenflag directions of the Weyl tensor are spanned by the coordinate vector fields.
\item The functions $t,y,z$ are LCWs.
\item The function $x$ is not a LCW.
 \end{itemize}
\end{thm}
\begin{proof}
 The non-zero Christoffel symbols for this metric are
\[
\Gamma^x_{yy} = -1/2, \quad
\Gamma^x_{zz} = -x , \quad
\Gamma^y_{xy}=\Gamma^y_{yx} = \frac{1}{2x}, \quad
\Gamma^z_{xz}=\Gamma^z_{zx} = \frac{1}{x},
\]
the $(4,0)$ Riemann curvature tensor is
\[
\mathrm{Riem}\left(g\right) = 
\frac{1}{x} \left(\mathrm{d} x\wedge \mathrm{d} y \right)\otimes \left(\mathrm{d} x\wedge \mathrm{d} y \right) 
+ 2 x \left(\mathrm{d} y\wedge \mathrm{d} z\right)\otimes \left(\mathrm{d} y\wedge\mathrm{d} z \right),
\]
the Ricci tensor is
$$
\mathrm{Ric}\left(g\right) = 
\frac{1}{4 \, x^{2}} \mathrm{d} x\otimes \mathrm{d} x
-\frac{1}{4 \, x} \mathrm{d} y\otimes \mathrm{d} y
-\frac{1}{2} \mathrm{d} z\otimes \mathrm{d} z,
$$
and the scalar curvature is $-\frac{1}{2x^2}$.

We define the normalized vector fields $\hat{\partial}_y = \frac{\partial_y}{\sqrt{x}}$ and  $\hat{\partial}_z = \frac{\partial_z}{x}$.

In the orthonormal basis of $\Lambda^2 T_pU$
\[
\partial_t\wedge\partial_x, \quad \partial_t\wedge\hat{\partial}_y,\quad \partial_t\wedge\hat{\partial}_z,\quad \hat{\partial}_y\wedge\hat{\partial}_z, \quad\partial_x\wedge\hat{\partial}_z,\quad \partial_x\wedge\hat{\partial}_y
\]
the Weyl tensor of $g$ is diagonal, with a $3\times 3$ block repeated twice
$$
W=\left(\begin{array}{rr}
W_1 & \\
& W_1 
\end{array}\right),
\quad
W_1=\left(\begin{array}{ccc}
-\frac{5}{96 \, x^{2}} & 0 & 0 \\
0 & \frac{1}{96 \, x^{2}} & 0 \\
0 & 0 & \frac{1}{24 \, x^{2}}
\end{array}\right)
$$
The proof of Lemma \ref{lemma: classification of algebraic Weyl tensors} shows that the four eigenflag directions are spanned by $\partial_t$, $\partial_x$, $\partial_y$ and $\partial_z$.

It is obvious that $\partial_t$ is a unit parallel vector field, while $\partial_y$ and $\partial_z$ are unit parallel vector fields for the conformal metrics $\frac{1}{x} g$ and $\frac{1}{x^2} g$, respectively.
Thus, according to Theorem~\ref{DKSU07}, the coordinate functions $t$, $y$ and $z$ are LCWs.

In order to check if $\langle\partial_x\rangle$ is a conformal factor, we want to apply our Theorem~\ref{prop: characterization of a conformal product factor}.
The second fundamental form of $\{\partial_x\}^\perp$ in the orthonormal basis $\{\partial_t, \hat{\partial}_y, \hat{\partial}_z\}$ is:
$$
\left(\begin{array}{ccc}
0 & 0 & 0 \\
0 & -\frac{1}{2 \, x} & 0 \\
0 & 0 & -\frac{1}{x}
\end{array}\right)
$$
which is not a multiple of the identity.
This means that $\{\partial_x\}^\perp$ is not umbilical, and we deduce that $\langle\partial_x\rangle$ is not a conformal factor.
\end{proof}

\section{A metric not conformal to a product, but with Weyl tensor of constant type C}\label{section: example FH}

In this example we show a explicit metric whose Weyl tensor has type C at all points in an open set, but which is not conformal to a product.
After that, we show how to find out all its LCWs.

\begin{thm}
 Let $U$ be any open subset of  $\{(t,x,y,z)\in\R^4:x>0\}$.
 The metric 
\begin{equation}\label{metric with Weyl tensor of type C}
g = \left(\begin{array}{cccc}
1 & 0 & 0 & 0 \\
0 & 1 & 0 & 0 \\
0 & 0 & x^3 & 0 \\
0 & 0 & 0 & 1/x
\end{array}\right)
\end{equation} 
 on $U$ has Weyl tensor of type C at every point, but is not conformal to a product of surfaces. Moreover there are only three LCW which are the coordinate
 functions $t,y,z$. 
\end{thm}
\begin{proof}
The non-zero Christoffel symbols for this metric are
$$
\Gamma^x_{yy} = -\frac{3}{2}x^2 , \quad
\Gamma^x_{zz} = \frac{1}{2\,x^2},  \quad
\Gamma^y_{xy}=\Gamma^y_{yx} = \frac{3}{2\,x},\quad
\Gamma^z_{xz}=\Gamma^z_{zx} = -\frac{1}{2\,x},
$$
the $(4,0)$-Riemann curvature tensor is
\begin{multline}
\mathrm{Riem}\left(g\right) =
- 3 x \left(\mathrm{d} x\wedge \mathrm{d} y \right)\otimes \left(\mathrm{d} x\wedge \mathrm{d} y \right) \\
- \frac{3}{x^3} \left(\mathrm{d} x\wedge \mathrm{d} z\right)\otimes \left(\mathrm{d} x\wedge\mathrm{d} z \right) 
-3 \left(\mathrm{d} y\wedge \mathrm{d} z\right)\otimes \left(\mathrm{d} y\wedge \mathrm{d} z \right),
\end{multline}
the Ricci tensor is
$$
\mathrm{Ric}\left(g\right) = -\frac{3}{2 \, x^{2}} \mathrm{d} x\otimes \mathrm{d} x,
$$
and the scalar curvature is $-\frac{3}{2x^2}$.

We define the normalized vector fields $\hat{\partial}_y = x^{-3/2}\partial_y$ and  $\hat{\partial}_z = \sqrt{x}\partial_z$.

In the orthonormal basis of $\Lambda^2 T_pU$
\[
\partial_t\wedge\partial_x, \quad \partial_t\wedge\hat{\partial}_y,  \quad \partial_t\wedge\hat{\partial}_z, \quad  \hat{\partial}_y\wedge\hat{\partial}_z, \quad  \partial_x\wedge\hat{\partial}_z, \quad  \partial_x\wedge\hat{\partial}_y
\]
the Weyl tensor of $g$ is diagonal, with a $3\times 3$ block repeated twice
$$
W=\left(\begin{array}{rr}
W_1 & \\
& W_1
\end{array}\right)
\quad
W_1=\left(\begin{array}{rrr}
\frac{1}{8 \, x^{2}} & & \\
 & -\frac{1}{16 \, x^{2}} & \\
 & & -\frac{1}{16 \, x^{2}}
\end{array}\right)
$$

The proof of \ref{lemma: classification of algebraic Weyl tensors} shows that the two $2$-planes distributions of eigenflag vectors are spanned by $\{\partial_t, \partial_x\}$ and $\{\partial_y, \partial_z\}$.
We deduce from Lemma \ref{lemma: Weyl operator of a product of surfaces} that our manifold can only be conformal to a product of two dimensional manifolds tangent to the planes $\langle\partial_t, \partial_x\rangle$ and $\langle\partial_y, \partial_z\rangle$.

However, Theorem~\ref{prop: characterization of a conformal product factor} implies that if our metric were conformal to a product metric, the distribution $\langle\partial_y, \partial_z\rangle$ would be umbilical.
But in that case, it follows from Definition \ref{def: umbilic distribution} that the following two numbers should be the same
\begin{equation}\label{first factor in example FH is not umbilical}
\begin{array}{rcl}
g(\nabla_{\hat{\partial}_y}\hat{\partial}_y,\partial_x) &=& -\frac{3}{2x},\\
g(\nabla_{\hat{\partial}_z}\hat{\partial}_z,\partial_x) &=& \frac{1}{2x} ,
\end{array}
\end{equation}
and we conclude that the metric is not conformal to a product metric.

Thus, the metric in the above Theorem is not covered by our Theorem~\ref{thm:four dimensions}.

It is obvious that $\partial_t$ is a parallel vector field, while $\partial_y$ and $\partial_z$ are unit parallel vector fields for the conformal metrics $\frac{1}{x^3} g$ and $x g$, respectively.
Thus, according to Theorem~\ref{DKSU07}, the coordinate functions $t$, $y$ and $z$ are LCWs.

In order to check if $\langle\partial_x\rangle$ is a conformal factor, we want to apply our Theorem~\ref{prop: characterization of a conformal product factor}.
The second fundamental form of $\{\partial_x\}^\perp$ in the orthonormal basis $\{\partial_t, \hat{\partial}_y, \hat{\partial}_z\}$ is:
$$
\left(\begin{array}{rrr}
0 & 0 & 0 \\
0 & -\frac{3}{2 \, x} & 0 \\
0 & 0 & \frac{1}{2 \, x}
\end{array}\right)
$$
which is not a multiple of the identity.
This means that $\{\partial_x\}^\perp$ is not umbilical, and we deduce that $\langle\partial_x\rangle$ is not a conformal factor.

The analysis of the Weyl tensor shows that any possible conformal factor must be contained in either $\langle\partial_t, \partial_x\rangle$ or $\langle\partial_y, \partial_z\rangle$.
Let us push Theorem~\ref{prop: characterization of a conformal product factor} a little bit further to find all of them. 

Let $X$ be a unit vector field in $\langle\partial_t, \partial_x\rangle$.
$X$ can be written in the form
$$
X = \cos(\alpha)\partial_t + \sin(\alpha)\partial_x
$$
for a real valued function $\alpha:M\rightarrow \R$.

If $X$ spans a conformal factor for some $\alpha$, then $\langle X\rangle^\perp$ is an umbilical distribution.
$X\in\langle\partial_t, \partial_x\rangle$ implies $\langle\partial_y, \partial_z\rangle\subset\langle X\rangle^\perp$, so in particular the form
$$
Z\rightarrow g(\nabla_Z Z, X)
$$
for $Z\in \langle\partial_y, \partial_z\rangle$ must be a multiple of the identity.
This tensor can be written as a linear combination
$$
g(\nabla_Z Z, X) = \cos(\alpha) g(\nabla_Z Z, \partial_t) + \sin(\alpha) g(\nabla_Z Z, \partial_x)
$$
but we notice that the first summand is zero
$$
Z\rightarrow g(\nabla_Z Z, \partial_t) = 0
$$
and we saw in (\ref{first factor in example FH is not umbilical}) that the matrix of $Z\rightarrow g(\nabla_Z Z, \partial_x)$ in the orthonormal basis $\{\hat{\partial}_y,\hat{\partial}_z\} $ is not a multiple of the identity.
Thus, the only combination of them that produces a multiple of the identity is $\cos(\alpha)=1, \sin(\alpha)=0$.

The same trick will not help us decide whether there are conformal factors of dimension $1$ contained in $\langle\partial_y, \partial_z\rangle$.
Instead, we define
$$
Z = \cos(\alpha)\hat{\partial}_y + \sin(\alpha)\hat{\partial}_z
$$
and compute the second fundamental form of $Z^\perp$ in the basis 
\[
\partial_t, \quad \partial_x, \quad -\sin(\alpha)\hat{\partial}_y + \cos(\alpha)\hat{\partial}_z.
\]
We remark that for some choices of $\alpha$ the distribution $Z^\perp$ is not integrable, which is why the matrix 
$$
II = \begin{pmatrix}
0& 0& 0 \\
0 & 0 &  - 2 \cos (\alpha) \sin (\alpha)\\
  - \partial_t \alpha & - \partial_x \alpha & \frac{x^2 \cos (\alpha)
  \partial_z \alpha- \sin (\alpha) \partial_y \alpha}{x^{3/2}}
\end{pmatrix}
$$
is not always symmetric.
If $Z^\perp$ is umbilical for some choice of $\alpha$, then the above matrix must be a multiple of the identity, and hence it would vanish identically. In particular, $- 2 \cos (\alpha) \sin (\alpha)$ is zero, and since $\alpha$ is continuous, we only have two choices: $Z=\hat{\partial}_y$ and $Z= \hat{\partial}_z$.
\end{proof}

\begin{rem}
The above example, shows how even if the analysis of the Weyl tensor yields an infinite  number of candidates to be one dimensional   conformal factors, a  use of Theorem~\ref{prop: characterization of a conformal product factor}  allows to rule out the fake ones.  Notice that in this example, it happened that for the false candidates, $Z^{\perp}$ was not umbilical.  In the rare event (we do not know of any example) that for false candidates
$Z^{\perp}$ was umbilical, the conditions on $H_1+H_2$ being a gradient field would rule out those eigenflags directions not arising from LCW. 
\end{rem}


\begin{thebibliography}{11}

\bibitem{AFGR} Angulo-Ardoy P., Faraco D., Guijarro L., Ruiz A., \emph{Obstructions to the existence of limiting Carleman weights}, (Preprint arXiv:1411.4887), to appear in Analysis and P.D.E

\bibitem{TransversalityPaper} Angulo-Ardoy, P. \emph{On the set of metrics without local limiting Carleman weights}, (Preprint arXiv:1509.02127)

%

%

%

\bibitem{CaroSalo14}
 Caro, Pedro;   Salo Mikko.    \emph{ Stability of the Calder\'on problem in admissible geometries}  (Preprint 	arXiv:1404.6652  )
\bibitem{DC} Do Carmo, Manfredo P.,  \emph{Riemannian geometry}. Mathematics: Theory and Applications. Birkhäuser Boston, Inc., Boston, MA, 1992.

\bibitem{CaroRogers} Caro, Pedro; Rogers Keith M Global uniqueness for the Calder\'on problem with Lipschitz conductivities. http://arxiv.org/abs/1411.8001
%
%

\bibitem{DKSU07} Dos Santos Ferreira, David; Kenig, Carlos E.; Salo, Mikko; Uhlmann, Gunther, \emph{Limiting Carleman weights and anisotropic inverse problems}, Invent. Math. 178 (2009), no. 1, 119--171.

\bibitem{DKS13} Dos Santos Ferreira, David; Kenig, Carlos E.; Salo, Mikko  \emph{Determining an unbounded potential from Cauchy data in admissible geometries}, Comm. PDE 38 (2013), no. 1, 50-68.

\bibitem{DKLS14}Dos Santos Ferreira, David; Kurylev Yaroslav; Lassa, Matti; Salo, Mikko 
\emph{The Calderon problem in transversally anisotropic geometries} (Preprint)

\bibitem{HabermanTataru13}  Haberman, Boaz; Tataru, Daniel Uniqueness in Calder\'on's problem with Lipschitz conductivities. Duke Math. J. 162 (2013), no. 3, 496?516.

\bibitem{Haberman15}  Haberman, Boaz Uniqueness in Calder\'on's problem for conductivities with unbounded gradient. Comm. Math. Phys. 340 (2015), no. 2, 639?659
%

%

\bibitem{KSU10}  Kenig, Carlos E.;  Salo, Mikko; Uhlmann, Gunther   \emph{ Reconstructions from boundary measurements on admissible manifolds.}	 Inverse Probl. Imaging 5 (2011), no. 4, 859-877

%

%

%

%
%

%

%


%

%

\bibitem{SaloLN} Salo, Mikko, \emph{The Calder\'on problem on Riemannian manifolds}, Chapter in Inverse Problems and Applications: Inside Out II (edited by G. Uhlmann), MSRI Publications, Cambridge University Press, 2012.

\bibitem{Salo13} Salo, Mikko,  \emph{The Calder\'on problem on Riemannian manifolds. Inverse problems and applications: inside out. II}, 167--247, Math. Sci. Res. Inst. Publ., 60, Cambridge Univ. Press, Cambridge, 2013.

%
%
%
\bibitem{SylvesterUhlman87} J. Sylvester and G. Uhlmann, A global uniqueness theorem
for an  inverse boundary value problem. Annals of Math. {\bf 125},
1987), 153--169.
%

\bibitem{Tojeiro} Tojeiro, Ruy, \emph{Conformal de Rham decomposition of Riemannian manifolds}. Houston J. Math. 32, No. 3, 725-743 (2006).

%
%
%

\end{thebibliography}
\end{document}